\documentclass[leqno]{siamltex}

\usepackage{amsfonts}
\usepackage{amsmath}
\usepackage{graphicx}

\usepackage{todonotes}

\title{Local well-posedness of the two-layer shallow water model with free surface\thanks{This 
        work was supported by the French Naval Hydrographic and Oceanographic Service.}}

\author{Ronan~Monjarret\thanks{Institut de Math\'ematiques de Toulouse, Universit\'e Toulouse III - Paul Sabatier, 118 route de Narbonne, 31062 Toulouse Cedex 9, France ({\tt ronan.monjarret@math.univ-toulouse.fr}).}}

\begin{document}

\maketitle

\begin{abstract}
In this paper, we address the question of the hyperbolicity and the local well-posedness of the two-layer shallow water model, with free surface, in two dimensions. We first provide a general criterion that proves the symmetrizability of this model, which implies hyperbolicity and local well-posedness in $\mathcal{H}^s(\mathbb{R}^2)$, with $s>2$. Then, we analyze rigorously the eigenstructure associated to this model and prove a more general criterion of hyperbolicity and local well-posedness, under weak density-stratification assumption. Finally, we consider a new {\em conservative} two-layer shallow water model, prove the hyperbolicity and the local well-posedness and rely it to the basic two-layer shallow water model.
\end{abstract}



\pagestyle{myheadings}
\thispagestyle{plain}
\markboth{Ronan Monjarret}{Two-layer shallow water model with free surface}

\section{Introduction}

We consider two immiscible, homogeneous, inviscid and incompressible superposed fluids, with no surface tension; the pressure is assumed to be hydrostatic, constant at the free surface and continuous at the internal surface. Moreover, the shallow water assumption is considered: there exist vertical and horizontal characteristic lengths and the vertical one is assumed much smaller than the horizontal one.

\noindent
For more details on the derivation of these equations, see \cite{desaint1871theorie}, \cite{pedlosky1982geophysical} and \cite{gill1982atmosphere} for the one-layer model; \cite{long1956long} for the two-layer model with rigid lid; \cite{schijf1953theoretical}, \cite{ovsyannikov1979two} and \cite{liska1997analysis} for the two-layer model with free surface. In the {\em curl}-free case, these models have been rigorously obtained as an asymptotic model of the three-dimensional {\em Euler} equations, under the shallow water assumption, in \cite{alvarez2007nash} for the one-layer model with free surface and in \cite{duchene2010asymptotic} for the two-layer one. With no assumption on the vorticity, it has been obtained, only in the one layer case, in \cite{Castrowellposed2014}.

\noindent
The aim of this paper is to obtain criteria of symmetrizability and hyperbolicity of the two-layer shallow water model, in order to insure the local well-posedness of the associated initial value problem.

\noindent
{\em Outline:} In this section, the model is introduced. In the $2^{\mathrm{nd}}$ one, useful definitions are reminded and a sufficient condition of hyperbolicity and local well-posedness in $\mathcal{H}^s(\mathbb{R}^2)$, is given. In the $3^{\mathrm{rd}}$ section, the hyperbolicity of the model is exactly characterized in one and two dimensions. In the $4^{\mathrm{th}}$ one, asymptotic analysis is performed, in order to deduce a new criterion of local well-podeness in $\mathcal{H}^s(\mathbb{R}^2)$. Finally, in the last section, after reminding the horizontal vorticity, a new model is introduced: benefits of this model are explained, local well-posedness, in $\mathcal{H}^s(\mathbb{R}^2)$, is proved and links, with the two-layer shallow water model, are justified.

\subsection{Governing equations}

\noindent
The $i^{\mathrm{th}}$ layer of fluid, $i \in \{1,2\}$, has a constant density $\rho_i$, a depth-averaged horizontal velocity $\boldsymbol{u_i}(t,X):={}^{\top} (u_i(t,X),v_i(t,X))$ and a thickness designated by $h_i(t,X)$, where $t$ denotes the time and $X:=(x,y)$ the horizontal cartesian coordinates, as drawn in figure \ref{2figure2layer}.

\noindent
The governing equations of the two-layer shallow water model with free surface are given by one mass conservation for each layer:
\begin{equation}
\frac{\partial h_i}{\partial t} + {\bf \nabla} {\bf \cdot} (h_i \boldsymbol{u_i}) = 0,\ i \in \{1,2\},
\label{2massconservation}
\end{equation}
and an equation on the depth-averaged horizontal velocity in each layer:
\begin{equation}
\frac{\partial \boldsymbol{u_i}}{\partial t}+(\boldsymbol{u_i} {\bf \cdot} {\bf \nabla}) \boldsymbol{u_i} + {\bf \nabla} P_i - f \boldsymbol{u_i}^{\bot} = 0,\ i \in \{1,2\},
\label{2momentumconservation}
\end{equation}
where $\boldsymbol{u_i}^{\bot}:={}^{\top} (v_i,-u_i)$, $P_i:=g\big( b+\sum_{k=1}^2 \alpha_{i,k} h_k \big)$ is the fluid pressure, with $g$ the gravitational acceleration, $b$ the bottom topography, $f$ the {\em Coriolis} parameter and $(\alpha_{i,k})_{(i,k) \in [\![1,n]\!]}$ given by

\begin{equation}
 \alpha_{i,k}=\left\{
 \begin{array}{ll}
 \frac{\rho_1}{\rho_2}, & \mathrm{if}\ k = i-1 = 1,\\[1pt]
 1, & \mathrm{otherwise}.
 \end{array}
 \right.
 \nonumber
\end{equation}

\begin{figure}[ht]
\centering
\includegraphics[width=7cm,height=50mm]{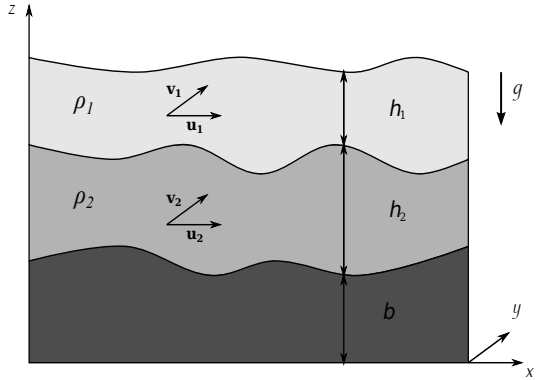}
\caption{Configuration of the two-layer shallow water model with free surface}
\label{2figure2layer}
\end{figure}

\noindent
The multi-layer shallow water model with free surface describes fluids such as the ocean: the evolution of the density can be assumed piecewise-constant, the horizontal characteristic length is much greater than the vertical one and the pressure can be expected only dependent of the height of fluid. The two-layer model is a simplified case, where we consider the density has only two values. This model describes well the straits of {\em Gibraltar}, where the Mediterrenean sea meets the Atlantic ocean.

\noindent
By introducing the vector
\begin{equation}
{\bf u}:={}^{\top} (h_1,h_2,u_1,u_2,v_1,v_2),
\label{2u}
\end{equation}
and $\gamma:=\frac{\rho_1}{\rho_2}$, the system (\ref{2massconservation}-\ref{2momentumconservation}) can be written as
\begin{equation}
\frac{\partial {\bf u}}{\partial t} + \mathsf{A}_x({\bf u}) \frac{\partial {\bf u}}{\partial x} + \mathsf{A}_y ({\bf u}) \frac{\partial {\bf u}}{\partial y} + {\bf b}({\bf u})=0,
\label{2systemmultilayer}
\end{equation}
where  $\mathsf{A}_x({\bf u})$, $ \mathsf{A}_y({\bf u})$ and ${\bf b}({\bf u})$ are defined by
\begin{equation}
 \mathsf{A}_x({\bf u}) :=  \left[
\begin{array}{cccccc}
 u_1 & 0 & h_1 & 0 & 0 & 0 \\
 0 & u_2 & 0 & h_2 & 0 & 0 \\
 g & g & u_1 & 0 & 0 & 0  \\
 \gamma g & g & 0 & u_2 & 0 & 0 \\
 0 & 0 & 0 & 0 & u_1 & 0 \\
 0 & 0 & 0 & 0 & 0 & u_2 \\
\end{array}
\right],\
\mathsf{A}_y({\bf u}) :=  \left[
\begin{array}{cccccc}
 v_1 & 0 & 0 & 0 & h_1 & 0 \\
 0 & v_2 & 0 & 0 & 0 & h_2 \\
 0 & 0 & v_1 & 0 & 0 & 0 \\
 0 & 0 & 0 & v_2 & 0 & 0 \\
 g & g & 0 & 0 & v_1 & 0  \\
 \gamma g & g & 0 & 0 & 0 & v_2 \\
\end{array}
\right],
\label{2Axy}
\end{equation}
\begin{equation}
{\bf b}({\bf u}):= {}^{\top} \left(0,0,-f v_1+g \frac{\partial b}{\partial x},-f v_2 + g \frac{\partial b}{\partial x},f u_1 + g \frac{\partial b}{\partial y},f u_2+g \frac{\partial b}{\partial y}\right).
\label{2botom}
\end{equation}

\subsection{Rotational invariance} As the two-layer shallow water model with free surface is based on physical partial differential equations, it is predictable that it verifies the so-called rotational invariance:
\begin{equation}
\mathsf{A}({\bf u},\theta):=\cos(\theta) \mathsf{A}_x({\bf u}) + \sin(\theta) \mathsf{A}_y({\bf u})
\label{2Autheta}
\end{equation}
depends only on the matrix $\mathsf{A}_x$ and the parameter $\theta$. Indeed, there is the following relation:
\begin{equation}
\forall ({\bf u},\theta) \in \mathbb{R}^6 \times [0,2\pi],\ \mathsf{A}({\bf u},\theta)=\mathsf{P}(\theta)^{\mathsf{-1}} \mathsf{A}_x\left(\mathsf{P}(\theta){\bf u} \right) \mathsf{P}(\theta),
\label{2rotinv}
\end{equation}
with
\begin{equation}
\mathsf{P}(\theta) :=  \left[
\begin{array}{cccccc}
 1 & 0 & 0 & 0 & 0 & 0 \\
 0 & 1 & 0 & 0 & 0 & 0 \\
 0 & 0 & \cos(\theta) & 0 & \sin(\theta) & 0 \\
 0 & 0 & 0 & \cos(\theta) & 0 & \sin(\theta) \\
 0 & 0 & -\sin(\theta) & 0 & \cos(\theta) & 0  \\
 0 & 0 & 0 & -\sin(\theta) & 0 & \cos(\theta) \\
\end{array}
\right],
\label{2Ptheta}
\end{equation}
and an important point is $\mathsf{P}(\theta)^{\mathsf{-1}}={}^{\top}\mathsf{P}(\theta)$. The equality (\ref{2rotinv}) will permit to symplify the analysis of $\mathsf{A}({\bf u},\theta)$ to the analysis of $\mathsf{A}_x\left(\mathsf{P}(\theta){\bf u} \right)$.

\section{Well-posedness of the model: a $\boldsymbol{1^{\mathrm{st}}}$ criterion}\label{1sectionhyp}
In this section, we remind useful criteria of hyperbolicity and local well-posedness in $\mathcal{H}^s(\mathbb{R}^2)$. Connections between each one will be given and a $1^{\mathrm{st}}$ criterion of local well-posedness of the model (\ref{2systemmultilayer}) will be deduced.

\subsection{Hyperbolicity}
First, we give the definition, a useful criterion of hyperbolicity and an important property of hyperbolic model. We will consider the euclidean space $\mathcal{L}^2(\mathbb{R}^2),\|\cdot\|_{\mathcal{L}^2})$.

\begin{definition}[{\rm Hyperbolicity}]
Let ${\bf u}: \mathbb{R}^2 \mapsto \mathbb{R}^6$. The system {\em (\ref{2systemmultilayer})} is hyperbolic if and only
\begin{equation}
\exists\ c>0,\ \forall \theta \in [0,2\pi],\ \sup_{\tau \in \mathbb{R}}\| \exp \left( - i \tau \mathsf{A}({\bf u},\theta) \right) \|_{\mathcal{L}^2} \le c.
\label{2expAutheta}
\end{equation}
\label{2defcrithypuseful}
\end{definition}

\noindent
A useful criterion of hyperbolicity is in the next proposition.
\begin{proposition}
Let ${\bf u}: \mathbb{R}^2 \mapsto \mathbb{R}^6$. The model {\em (\ref{2systemmultilayer})} is hyperbolic if and only
\begin{equation}
\forall \theta \in [0,2\pi],\ \sigma \left( \mathsf{A}({\bf u},\theta) \right) \subset \mathbb{R}.
\label{2spectrumAutheta}
\end{equation}
\label{2propcrithypuseful}
\end{proposition}

\begin{proposition}
Let ${\bf u}: \mathbb{R}^2 \mapsto \mathbb{R}^6$ a constant function. If the model {\em (\ref{2systemmultilayer})} is hyperbolic, then the {\em Cauchy} problem, associated with the linear system
\begin{equation}
\frac{\partial {\bf v}}{\partial t} + \mathsf{A}_x({\bf u}) \frac{\partial {\bf v}}{\partial x} + \mathsf{A}_y ({\bf u}) \frac{\partial {\bf v}}{\partial y}=0,
\label{2linearprobl}
\end{equation}
and the initial data ${\bf v^0} \in \mathcal{L}^2(\mathbb{R}^2)^6$, is locally well-posed in $\mathcal{L}^2(\mathbb{R}^2)$ and the unique solution ${\bf v}$ is such that
\begin{equation}
\left\{
\begin{array}{l}
\forall\ T >0,\ \exists\ c_T > 0,\ \sup_{t \in [0,T]} \|{\bf v}(t)\|_{\mathcal{L}^2} \le c_T \|{\bf v^0}\|_{\mathcal{L}^2},\\
{\bf v} \in \mathcal{C}(\mathbb{R}_+;\mathcal{L}^2(\mathbb{R}^2))^6
\end{array}
\right.
\label{2regularityuhyp}
\end{equation}
\label{2defhyp}
\end{proposition}

\noindent
{\em Remark:} More details about the hyperbolicity in \cite{Serre1996systemes}.

\subsection{Symmetrizability}
In order to prove the local well-posedeness of the model (\ref{2systemmultilayer}), in $\mathcal{H}^s(\mathbb{R}^2)$, we give below a useful criteria.

\begin{definition}[{\rm Symmetrizability}]
Let ${\bf u} \in \mathcal{H}^s(\mathbb{R}^2)^6$. If there exists a $\mathcal{C}^{\infty}$ mapping $\mathsf{S}: \mathcal{H}^s(\mathbb{R}^2)^6 \times [0,2\pi] \rightarrow \mathcal{M}_{6}(\mathbb{R})$ such that for all $\theta \in [0,2\pi]$,
\begin{romannum}
\item $\mathsf{S}({\bf u},\theta)$ is symmetric,
\item $\mathsf{S}({\bf u},\theta)$ is positive-definite,
\item $\mathsf{S}({\bf u},\theta) \mathsf{A}({\bf u},\theta)$ is symmetric,
\end{romannum}
then, the model {\em (\ref{2systemmultilayer})} is said symmetrizable and the mapping $\mathsf{S}$ is called a symbolic-symmetrizer.
\label{2defsymm}
\end{definition}

\begin{proposition}
Let ${\bf u^0} \in \mathcal{H}^s(\mathbb{R}^2)^6$. If the model {\em (\ref{2systemmultilayer})} is symmetrizable, then the {\em Cauchy} problem, associated with {\em (\ref{2systemmultilayer})} and initial data ${\bf u^0}$, is locally well-posed in $\mathcal{H}^s(\mathbb{R}^2)$, with $s>2$. Furthermore, there exists $T>0$ such that the unique solution ${\bf u}$ verifies
\begin{equation}
\left\{
\begin{array}{l}
{\bf u} \in \mathcal{C}^1([0,T] \times \mathbb{R}^2)^6,\\
{\bf u} \in \mathcal{C}([0,T];\mathcal{H}^s(\mathbb{R}^2))^6 \cap \mathcal{C}^1([0,T];\mathcal{H}^{s-1}(\mathbb{R}^2))^6.
\end{array}
\right.
\label{2strongsol}
\end{equation}
\label{2propsymm}
\end{proposition}

\noindent
{\em Remark:} The proof of the last proposition is in \cite{benzoni2007multi}, for instance.

\noindent
In this paper, the model (\ref{2massconservation}--\ref{2momentumconservation}) is expressed with the variables $(h_i,u_i)$ with $i \in \{1,2\}$. However, we could have worked with the unknowns $h_i$ and $q_i:=h_i u_i$, as it is well-known this quantities are conservative in the one-dimensional case. However, in the particular case of the two-layer shallow water model with free surface, it is not true. Indeed, the one-dimensional model expressed in $(h_i,u_i)$ is conservative and the one in $(h_i,q_i)$ is not.

\noindent
As it was noticed in \cite{Serre1996systemes}, if the model is conservative, there exists a natural symmetrizer: the hessian of the energy of the model. This energy is defined, modulo a constant, by:
\begin{equation}
e_1:=\frac{1}{2}\gamma h_1 \left(u_1^2+ g (h_1+2 h_2)\right)+\frac{1}{2} h_2 \left( u_2^2+g h_2\right).
\label{2energy}
\end{equation}
As the model (\ref{2massconservation}--\ref{2momentumconservation}), in one dimension and variables $(h_i,u_i)$, is conservative, it is straightforward the hessian of $e_1$ is a symmetrizer of the one-dimensional model. However, it is not anymore a symmetrizer with the non-conservative variables $(h_i,h_i u_i)$. This is why the analysis, in this paper, is performed with variables $(h_i,u_i)$.

\noindent
{\em Remarks:} 1) In all this paper, the parameter $s \in \mathbb{R}$ is assumed such that
\begin{equation}
s > 1+ \frac{d}{2},
\label{1params}
\end{equation}
where $d:=2$ is the dimension. 2) The criterion (\ref{2spectrumAutheta}) is a necessary and sufficient condition of hyperbolicity, whereas the symmetrizability is only a sufficient condition of local well-posedness in $\mathcal{H}^s(\mathbb{R}^2)$.

\subsection{Connections between hyperbolicity and symmetrizability}
In this subsection, we do not formulate all the connections between the hyperbolicity and the local well-posedness in $\mathcal{H}^s(\mathbb{R}^2)$ but only the useful ones for this paper.

\begin{proposition}
If the model {\em (\ref{2systemmultilayer})} is symmetrizable, then it is hyperbolic.
\label{2symimpliqhyp}
\end{proposition}

\noindent
{\em Remark:} See \cite{benzoni2007multi} or \cite{Serre1996systemes} for more details.

\begin{proposition}
Let ${\bf u^0} \in \mathcal{H}^s(\mathbb{R}^2)^6$ such that the model is hyperbolic and
\begin{equation}
\forall (X,\theta) \in \mathbb{R}^2 \times [0,2\pi],\ \mathrm{the}\ \mathrm{matrix}\ \mathsf{A}({\bf u^0}(X),\theta)\ \mathrm{is}\ \mathrm{diagonalizable}.
\label{2diagoAhyp}
\end{equation}
Then, the model {\em (\ref{2systemmultilayer})} is symmetrizable and the unique solution verifies the conditions {\em (\ref{2strongsol})}.
\label{2diagimpliqsym}
\end{proposition}
\begin{proof}
Let $\mu \in \sigma(\mathsf{A}({\bf u^0},\theta))$, we denote $\mathsf{P}^{\mu}({\bf u^0},\theta)$ the projection onto the $\mu$-eigenspace of $\mathsf{A}({\bf u^0},\theta)$. One can construct a symbolic symmetrizer:
\begin{equation}
\mathsf{S_1}({\bf u^0},\theta):= \sum_{\mu \in \sigma(\mathsf{A}({\bf u^0},\theta))} {}^{\top} \mathsf{P}^{\mu}({\bf u^0},\theta) \mathsf{P}^{\mu}({\bf u^0},\theta).
\label{2symprojdef}
\end{equation}
Then, $\mathsf{S_1}({\bf u^0},\theta)$ verifies conditions of the proposition \ref{2propsymm} because $\mathsf{A}({\bf u^0},\theta)$ is diagonalizable and $\sigma(\mathsf{A}({\bf u^0},\theta)) \subset \mathbb{R}$. Then, proposition \ref{2propsymm} implies the local well-posedness of the model (\ref{2systemmultilayer}), in $\mathcal{H}^s(\mathbb{R}^2)$, and there exists $T>0$ such that conditions (\ref{2strongsol}) are verified.\end{proof}

To conclude, the analysis of the eigenstrusture of $\mathsf{A}({\bf u},\theta)$ is a crucial point, in order to provide its diagonalizability. Moreover, it provides also the characterization of the {\em Riemann} invariants (see \cite{smoller1983shock}), which is an important benefit for numerical resolution.

\noindent
{\em Remark:} This proposition was proved in \cite{taylor1996partial}, in the particular case of a strictly hyperbolic model ({\em i.e.} all the eigenvalues are real and distinct).

\subsection{A criterion of symmetrizability of the two-layer shallow water model} Acconrding to the proposition \ref{2symimpliqhyp}, the symmetrizability implies the hyperbolicity. Then, we give a rough criterion of symmetrizability to insure the local well-posedness in $\mathcal{H}^s(\mathbb{R}^2)$ and $\mathcal{L}^2(\mathbb{R}^2)$.

\begin{theorem}
Let $\gamma \in ]0,1[$ and ${\bf u^0} \in \mathcal{H}^s(\mathbb{R}^2)^6$ such that
\begin{equation}
\left\{
\begin{array}{l}
\inf_{X \in \mathbb{R}^2} h_1^0(X) > 0,\ \inf_{X \in \mathbb{R}^2} h_2^0(X) > 0,\\
\inf_{X \in \mathbb{R}^2} (1-\gamma) g h_2^0(X) - (u_2^0(X)-u_1^0(X))^2-(v_2^0(X)-v_1^0(X))^2 > 0.
\end{array}
\right.
\label{2condwellposed}
\end{equation}
Then, the {\em Cauchy} problem, associated with {\em (\ref{2systemmultilayer})} and the initial data ${\bf u^0}$, is hyperbolic, locally well-posed in $\mathcal{H}^s(\mathbb{R}^2)$ and the unique solution verifies conditions {\em (\ref{2strongsol})}.
\label{2condwellposedMLSW}
\end{theorem}
\begin{proof}
First, we prove the next lemma

\begin{lemma}
Let $\mathcal{S}$ be an open subset of $\mathcal{H}^s(\mathbb{R}^2)^6$ and $\mathsf{S_x}({\bf u})$ be a symmetric matrix such that $\mathsf{S_x}({\bf u}) \mathsf{A}_x({\bf u})$ is symmetric. If there exists ${\bf u^0} \in \mathcal{S}$ such that
\begin{equation}
\forall \theta \in [0,2\pi],\ \mathsf{S_x}(\mathsf{P}(\theta){\bf u^0})>0,
\label{2positivdefSp}
\end{equation}
then the system {\em (\ref{2systemmultilayer})}, with initial data ${\bf u^0}$, is locally well-posed in $\mathcal{H}^s(\mathbb{R}^2)$, hyperbolic and the unique solution verifies {\em (\ref{2strongsol})}.
\label{2wellposedrotinv}
\end{lemma}
\begin{proof}
Considering ${\bf u^0} \in \mathcal{S}$ such that for all $\theta \in [0,2\pi]$, $\mathsf{S_x}(\mathsf{P}(\theta){\bf u})$ is positive-definite, it is clear that $\mathsf{S}: ({\bf u},\theta) \mapsto \mathsf{P}(\theta)^{-1} \mathsf{S_x}(\mathsf{P}(\theta) {\bf u}) \mathsf{P}(\theta)$ verifies assumptions of proposition \ref{2propsymm}, with ${\bf u^0} \in \mathcal{H}^s(\mathbb{R}^2)^6$. Consequently, propositions \ref{2propsymm} and \ref{2symimpliqhyp} are verified.
\end{proof}

Then, to verify the lemma \ref{2wellposedrotinv}, we use a perturbation of the hessian of $e_1$ (which is a symmetrizer of the one-dimensional model, as we noticed it before):
\begin{equation}
\mathsf{S_x}({\bf u}) :=  \left[
\begin{array}{cccccc}
g \gamma  & g \gamma  & \gamma (u_1-u_0)  & 0 & 0 & 0 \\
g \gamma  & g & 0 & u_2-u_0 & 0 & 0 \\
 \gamma (u_1-u_0)  & 0 & \gamma h_1  & 0 & 0 & 0 \\
 0 & u_2-u_0 & 0 & h_2 & 0 & 0 \\
 0 & 0 & 0 & 0 & \gamma h_1  & 0 \\
 0 & 0 & 0 & 0 & 0 & h_2
\end{array}
\right],
\label{2Sx}
\end{equation}
where $u_0 \in \mathbb{R}$ is a parameter, which will be chosen in order to simplify the calculus. Then, it is clear that $\mathsf{S_x}({\bf u},u_0)$ and $\mathsf{S_x}({\bf u},u_0) \mathsf{A}_x({\bf u})$ are symmetric for all $({\bf u},u_0) \in \mathbb{R}^6 \times \mathbb{R}$. From now on $u_0$ is set as $u_0:=u_1$. Then, using the leading principal minors characterization of a positive-definite matrix (also known as {\em Sylvester}'s criterion), for $(X,\theta) \in \mathbb{R}^2 \times [0,2\pi]$, $\mathsf{S_x}(\mathsf{P}(\theta) {\bf u^0}(X)) > 0$ if and only if
\begin{equation}
\left\{
\begin{array}{l}
\gamma \in ]0,1[,\\
h_1^0(X) > 0,\ h_2^0(X) > 0,\\
(1-\gamma) g h_2^0(X) > \left(\cos(\theta)(u_2^0(X)-u_1^0(X))+\sin(\theta)(v_2^0(X)-v_1^0(X))\right)^2.
\end{array}
\right.
\label{2conddefinSx}
\end{equation}
Finally, as conditions (\ref{2conddefinSx}) must be verified for all $(X,\theta) \in \mathbb{R}^2 \times [0,2\pi]$ and remarking that for all $(\alpha,\beta) \in \mathbb{R}^2$,
\begin{equation}
\max_{\theta \in [0,2\pi]} \left(\cos(\theta)\alpha+\sin(\theta)\beta\right)^2 = \alpha^2 + \beta^2,
\label{2majorFtheta}
\end{equation}
then, according to the lemma \ref{2wellposedrotinv}, the system is locally well-posed in $\mathcal{H}^s(\mathbb{R}^2)$ and hyperbolic, under conditions (\ref{2condwellposed}), with $\gamma \in ]0,1[$.
\end{proof}

\noindent
{\em Remark:} The conditions (\ref{2conddefinSx}) have already been found in \cite{duchene2010asymptotic}, in the {\em curl}-free case.

\noindent
To conclude, let $\gamma \in ]0,1[$, we define $\mathcal{S}^s_{\gamma} \subset \mathcal{H}^s(\mathbb{R}^2)^6$, an open subset of intial conditions such that the model (\ref{2systemmultilayer}) is symmetrizable:
conditions (\ref{2conddefinSx}) are verified:
\begin{equation}
\mathcal{S}^s_{\gamma}:= \left\{ {\bf u^0} \in \mathcal{H}^s(\mathbb{R}^2)^6 / {\bf u^0}\ \mathrm{verifies}\ \mathrm{conditions}\ {\rm (\ref{2condwellposed})}  \right\}.
\label{2O}
\end{equation}

\section{Exact set of hyperbolicity}
In the previous section, we proved the hyperbolicity of the {\em Cauchy} problem, associated with the system (\ref{2systemmultilayer}) and the initial data ${\bf u^0}$, if $\gamma \in ]0,1[$ and ${\bf u^0} \in \mathcal{S}^s_{\gamma}$. However, it was just a sufficient condition of hyperbolicity. The purpose of this section is to characterize the exact set of hyperbolicity of the system (\ref{2systemmultilayer}): $\mathcal{H}_{\gamma}$, defined by
\begin{equation}
\mathcal{H}_{\gamma}:= \left\{ {\bf u^0}: \mathbb{R}^2 \mapsto \mathbb{R}^6 / {\bf u^0}\ \mathrm{verifies}\ \mathrm{conditions}\ {\rm (\ref{2spectrumAutheta})} \right\}
\label{2defHgamma}
\end{equation}
To do so, we reduce the analysis of the spectrum of $\mathsf{A}({\bf u},\theta)$: $\sigma(\mathsf{A}({\bf u},\theta))$, to the one of $\mathsf{A}_x(\mathsf{P}(\theta){\bf u})$, using the rotational invariance (\ref{2rotinv}). In this section, the study is performed onto the spectrum of $\mathsf{A}_x({\bf u})$ and is deduced, afterwards, to $\mathsf{A}_x(\mathsf{P}(\theta){\bf u})$. As the characteristic polynomial of $\mathsf{A}_x({\bf u})$ is equal to $\det(\mathsf{A}_x({\bf u})-\mu \mathsf{I_{6}})=(\mu-u_1)(\mu-u_2) Q(\mu)$, where $Q(\mu)$ is a quartic, it is necessary to get an exact real roots criterion for quartic equations

\subsection{Real roots criterion for quartic equations}
Considering a quartic equation
\begin{equation}
R(\lambda) := a_4 \lambda^4 + a_3 \lambda^3 + a_2 \lambda^2 + a_1 \lambda + a_0 = 0,
\label{2quartic}
\end{equation}
where $(a_0,a_1,a_2,a_3,a_4) \in \mathbb{R}^4$ and $a_4 > 0$. We define the {\em Sylvester}'s matrix $\mathsf{M}:=[\mathsf{M}_{i,j}]_{(i,j)\in [\![1,4 ]\!]}$ by 
\begin{equation}
\sum_{i,j} \mathsf{M}_{i,j} X^{4-i} Y^{4-j}=\frac{R(X)R'(Y)-R(Y)R'(X)}{X-Y},
\label{2mij}
\end{equation}

\noindent
Then, according to {\em Sturm}'s theorem, the roots of (\ref{2quartic}) are all real if and only if the matrix $\mathsf{M}$ is positive-definite or negative-definite. Then, using the {\em Sylvester}'s criterion on $\mathsf{M}$ ({\em i.e.} $\mathsf{M}>0$ if and only if all the leading principal minors are strictly positive), it provides an exact criterion of hyperbolicity of the system (\ref{2systemmultilayer}).

\begin{proposition}
The roots of the quartic equation are all real if and only if
\begin{equation}
\forall k \in [\![1,3]\!],\ m_{k+1} m_k > 0 
\label{2minors}
\end{equation}
where $m_k$ is the $k^{th}$ leading principal minor of $\mathsf{M}$.
\label{2propminors}
\end{proposition}

\noindent
{\em Remark:} This general criterion is exactly the same given in \cite{fuller1981root} and \cite{jury1981positivity}.
\subsection{Scaling of the equation}
In order to rescale the equation $Q(\mu)=0$, undimensioned quantities are considered, assuming $h_1>0$:
\begin{equation}
\lambda:=\frac{\mu-u_1}{\sqrt{g h_1}},\  F_x:=\frac{u_2 - u_1}{\sqrt{g h_1}},\ F_y:=\frac{v_2 - v_1}{\sqrt{g h_1}},\ h:=\frac{h_2}{h_1}.
\label{2undimquant}
\end{equation}
It is straightforward $(\lambda, F_x,F_y,h) \in \mathbb{R} \times \mathbb{R} \times \mathbb{R} \times \mathbb{R}_+$.
Consequently, the equation $Q(\mu)=0$ is equivalent to $P(\lambda)=0$ with
\begin{equation}
P(\lambda):= \left[\lambda^2 - 1 \right] \left[ \left(\lambda -  F_x \right)^2 - h \right] - \gamma h.
\label{2PLambda}
\end{equation}

\noindent
According to (\ref{2mij}), the symmetric matrix $\mathsf{M}:=[\mathsf{M}_{i,j}]_{(i,j)\in [|1,4 |]}$ is here defined by
\begin{equation}
\left\{
\begin{array}{l}
\mathsf{M}_{1,1}=4\\
\mathsf{M}_{1,2}= -6  F_x \\
\mathsf{M}_{1,3}=  -2(1+h- F_x^2)\\
\mathsf{M}_{1,4}=  2  F_x\\
\mathsf{M}_{2,2}=  2(1+h+5  F_x^2) \\
\mathsf{M}_{2,3}=  -2  F_x (1-2h+2  F_x^2) \\
\mathsf{M}_{2,4}=  4 h(\gamma-1)\\
\mathsf{M}_{3,3}=  2(1+h^2+4  F_x^2+ F_x^4+2h(\gamma- F_x^2)) \\
\mathsf{M}_{3,4}=  -2  F_x(1+h(3\gamma-2)+2  F_x^2)\\
\mathsf{M}_{4,4}=  F_x^4+ F_x^2+h(1-2 F_x^2+\gamma( F_x^2-1)))+2(h^2(1-\gamma)
\end{array}
\right.
\label{2M}
\end{equation}

\noindent
{\em Remark:} As $m_1=4$, it is impossible $\mathsf{M}$ negative-definite and all the leading principal minors of $\mathsf{M}$ have to be strictly positive.

\subsection{Hyperbolicity in one dimension}
In this subsection, we give the exact criterion of real solutions for $P(\lambda)=0$ and deduce a general criterion of hyperbolicity of the model (\ref{2systemmultilayer}), in one dimension.
\begin{proposition}
There exist $(F_{crit}^+,F_{crit}^-) \in \mathbb{R}_+^2$, with $F_{crit}^+ \ge F_{crit}^- \ge 0$, such that the roots of $P(\lambda)$ are all real if and only if
\begin{equation}
\gamma \in ]0,1[\ \mathrm{and}\ |F_x| \in [0,F_{crit}^-[ \cup ]F_{crit}^+,+ \infty[.
\label{2rootP}
\end{equation}
\label{2propexactcrit}
\end{proposition}

\noindent
In order to prove this proposition, we evaluate the exact conditions (\ref{2minors}), to prove the proposition \ref{2propminors}, which is an easy consequence of the following lemmata.
\begin{lemma}
For all $(F_x,h,\gamma) \in \mathbb{R} \times \mathbb{R}_+^2$,
\begin{equation}
m_1>0,\ m_2>0.
\label{2delta12}
\end{equation}
\label{2lemmaDelta12}
\end{lemma}
\begin{proof}
According to the expression of $\mathsf{M}$, $m_1=4$ and $m_2=8(1+h)+4 F_x^2$. As $h$ is assumed strictly positive, the lemma \ref{2lemmaDelta12} is proved.
\end{proof}

\begin{lemma}
Let $(F_x,h,\gamma) \in \mathbb{R} \times \mathbb{R}_+^2$. Then $\gamma \in ]0,1[$ if and only if
\begin{equation}
m_3>0.
\label{2delta3}
\end{equation}
\label{2lemmaDelta3}
\end{lemma}
\begin{proof}
As the expression of $m_3$ is
\begin{equation}
m_3 = 8(1+h)  F_x^4 - 16 (1 - h(6 + \gamma) + h^2)  F_x^2 + 8(1+h)(1- 2 h(1 - 2 \gamma) + h^2),
\end{equation}
it is considered as a quadratic polynomial in $z:= F_x^2$, with main coefficient positive
\begin{equation}
m_3=p_3(z):=b_2 z^2+b_1 z+b_0.
\label{2m3quadratic}
\end{equation}
with $b_2:=8(1+h)$, $b_1:=- 16 (1 - h(6 + \gamma) + h^2)$ and $b_3:=8(1+h)(1- 2 h(1 - 2 \gamma) + h^2)$. Then, it is strictly positive if and only if one of the two following assertions is verified
\begin{equation}
\left\{
\begin{array}{l}
p_3^1(h)<0,\\
\mathrm{the}\ \mathrm{roots}\ \mathrm{of}\ p_3(z)\ \mathrm{are}\ \mathrm{all}\ \mathrm{strictly}\ \mathrm{negative}.
\end{array}
\right.
\label{1assertionm3pos}
\end{equation}
where $p_3^1$ is the discriminant of the quadratic $p_3$
\begin{equation}
m_3^1:=-256h(6h-\gamma-2)(h(2+\gamma)-6).
\label{2m31}
\end{equation}

\noindent
In the first case, noting that roots of $p_3^1(h)$ are $0$, $\frac{6}{2+\gamma}$ and $\frac{2+\gamma}{6}$, and as $h$ is assumed strictly positive, the discriminant is strictly negative if and only if $2+\gamma>0$ and $h \not\in [h_{crit}^-,h_{crit}^+]$ or if $2+\gamma<0$ and $h \in [h_{crit}^-,h_{crit}^+]$, with
\begin{equation}
\left\{
\begin{array}{l}
h_{crit}^-:=\min(\frac{6}{2+\gamma},\frac{2+\gamma}{6}),\\
h_{crit}^+:=\max(\frac{6}{2+\gamma},\frac{2+\gamma}{6}).
\end{array}
\right.
\label{2hcrit}
\end{equation}
As $\gamma$ is assumed positive, if $h \not\in [h_{crit}^-,h_{crit}^+]$ then $m_3 >0$, and if $h \in [h_{crit}^-,h_{crit}^+]$, the second assertion of (\ref{1assertionm3pos}) should be verified: the roots of $p_3(z):=b_2 z^2+b_1 z+b_0$ are all strictly negative, which is equivalent to
\begin{equation}
b_2 b_0 > 0\ \mathrm{and}\ \frac{b_1}{b_2}>0.
\label{2secondassertverified}
\end{equation}
As, $b_2 b_0=(8(1+h))^2(1-2h(1-2\gamma)+h^2) > 0$ if and only if $4\gamma(1-\gamma)>0$
\begin{equation}
b_2 b_0>0 \Leftrightarrow \gamma  \in ]0,1[.
\label{2b2b0}
\end{equation}
Moreover, for all $(h,\gamma) \in [h_{crit}^-,h_{crit}^+] \times \mathbb{R}_+$, one can check $\frac{b_1}{b_2}>0$. To conclude, $m_3>0$ if and only if $\gamma \in ]0,1[$ and the lemma \ref{2lemmaDelta3} is proved.
\end{proof}

\begin{lemma}
Let $(F_x,h,\gamma) \in \mathbb{R} \times \mathbb{R}_+^2$ such that $\gamma \in ]0,1[$. Then, there exist two positive real, denoted by $ F ^{\pm}_{crit}$, such that $ F ^{+}_{crit} \ge  F ^{-}_{crit} \ge 0$ and
\begin{equation}
m_4>0 \Leftrightarrow | F_x| \in [0, F ^{-}_{crit}[\ \cup\ ] F ^{+}_{crit},+\infty[.
\label{2delta4}
\end{equation}
\label{2lemmaDelta4}
\end{lemma}
\begin{proof}
Considering $m_4 = 16 h q(z)$ where
\begin{equation}
\begin{array}{rl}
q(z) :=& z^4 + (h+1)(\gamma-4)z^3 - \left(3( h^2+1)(\gamma - 2)-h(\gamma^2 -26 \gamma +4) \right)z^2 \\
	         & + (1+h) \left((h^2+1)(3 \gamma - 4)+h(-20 \gamma^2+10 \gamma+8) \right) z \\
	         & - (\gamma-1) \left( (h-1)^2 + 4 \gamma h \right)^2
\end{array}
\label{2T}
\end{equation}
We denote by $\{z_1,z_2,z_3,z_4\}$ the roots of $q$. If $\gamma \in ]0,1[$, then it is obvious that
\begin{equation}
\left\{
\begin{array}{l}
q(0)>0,\\
q\left((1+\sqrt{h})^2\right)<0,\\
\lim_{z \rightarrow + \infty} q(z)=+\infty
\end{array}
\right.
\label{2signq1}
\end{equation}
 and then, $q$ has, at least, two positive real roots . Moreover, as $\lim_{z \rightarrow - \infty} q(z)=+\infty$ and the product of the roots is positive if $\gamma \in ]0,1[$
\begin{equation}
z_1 z_2 z_3 z_4=(1-\gamma)\left((h-1)^2+ 4\gamma h\right)>0,
\label{2signqinf}
\end{equation}
and the two other roots are complex or have the same sign. However, if all the roots are real, the {\em Sylvester}'s criterion is necessarily verified for the quartic $q$. Then, the {\em Sylvester}'s matrix associated to $q$ is not positive-definite because $n_k$, the $k^{\mathrm{th}}$ leading principal minors, with $k \in [\![1,4]\!]$, are such that
\begin{equation}
\forall \gamma \in  ]0,1[,\
\left\{
\begin{array}{l}
n_1=4 >0,\\
n_4=-\left(\gamma h(h-1)\right)^2(27\gamma^2(1+h^2)-2h(2\gamma^3+3\gamma^2+96\gamma-128))^3<0
\end{array}
\right.
\label{2n1n3}
\end{equation}
Consequently, the proposition \ref{2propminors} is not verified and all the roots of $q$ are not real. Then, there are exactly two positive roots of $q$, denoted by $ F ^{\pm\ 2}_{crit}$. Finally, if $\gamma \not\in ]0,1[$, one can prove that $q(z)$ has only one positive root, but it is not vital for the exact criterion of hyperbolicity, as $m_3>0$ if and only if $\gamma \in ]0,1[$.
\end{proof}

\noindent
{\em Remark:} The critical quantities $ F ^{\pm\ 2}_{crit}$ are analytical functions of $h$ and $\gamma$. The existence of these quantities has been noticed numerically in \cite{castro2011numerical}.

\begin{theorem}
Let ${\bf u^0}: \mathbb{R}^2 \mapsto \mathbb{R}^6$. The system {\em (\ref{2systemmultilayer})}, in one dimension, with initial data ${\bf u^0}$, is hyperbolic if and only if 
\begin{equation}
\left\{
\begin{array}{l}
\gamma \in ]0,1[\\
\inf_{x \in \mathbb{R}} h_1^0(x)>0,\ \inf_{x \in \mathbb{R}} h_2^0(x) >0,\\
\forall x \in \mathbb{R},\ | F_x^0(x) | < F ^{-\ 0}_{crit}(x)\ \mathrm{or}\ |F_x^0(x) | > F ^{+\ 0}_{crit}(x),
\end{array}
\right.
\label{2hypcrit1d}
\end{equation}
with $ F ^{\pm}_{crit}$ defined in lemma {\em \ref{2lemmaDelta4}}.
\label{2hypcritthm1d}
\end{theorem}
\begin{proof}
In the one dimension case, the matrix $\mathsf{A}({\bf u},\theta)$ is reduced to $\mathsf{A}_x({\bf u})$. Consequently, applying directly proposition \ref{2propexactcrit}, the theorem \ref{2hypcritthm1d} is proved.
\end{proof}

\subsection{Hyperbolicity in two dimensions}
In this subsection, we can deduce from below an exact criterion of hyperbolicity of the model (\ref{2systemmultilayer}). The next lemma is a reformulation of the result mentioned in \cite{barros2008hyperbolicity}

\noindent
Then, we can prove the next theorem
\begin{theorem}
Let ${\bf u^0}: \mathbb{R}^2 \mapsto \mathbb{R}^6$. The system {\em (\ref{2systemmultilayer})}, with initial data ${\bf u^0}$, is hyperbolic if and only if 
\begin{equation}
\left\{
\begin{array}{l}
\gamma \in ]0,1[\\
\inf_{X \in \mathbb{R}^2} h_1^0(x)>0,\ \inf_{X \in \mathbb{R}^2} h_2^0(x) >0,\\
\inf_{X \in \mathbb{R}^2} F ^{-\ 0}_{crit}(X)^2-F_x^{0}(X)^2- F_y^{0}(X)^2 >0.
\end{array}
\right.
\label{2hypcrit2d}
\end{equation}
\label{2hypcritthm2d}
\end{theorem}
\begin{proof}
As it was mentioned in proposition \ref{2propcrithypuseful}, the hyperbolicity is insured if and only if the spectrum of $\mathsf{A}({\bf u^0},\theta)$ included in $\mathbb{R}$, for all $\theta \in [0,2\pi]$. Moreover, using the rotational invariance (\ref{2rotinv}), it is equivalent with the spectrum of $\mathsf{A}_x(\mathsf{P}(\theta){\bf u})$ is included in $\mathbb{R}$, for all $\theta \in [0,2\pi]$. Then, with proposition \ref{2propexactcrit}, it is obvious the system (\ref{2systemmultilayer}) is hyperbolic if and only if
\begin{equation}
\forall \theta \in [0,2\pi],\
\left\{
\begin{array}{l}
\gamma \in ]0,1[,\\
| F(\theta)| \in [0, F ^{-}_{crit}[\ \cup\ ] F ^{+}_{crit},+\infty[,
\end{array}
\right.
\label{2hyp2dim}
\end{equation}
with $ F(\theta):=\cos(\theta)  F_x + \sin(\theta)  F_y$. Because these conditions are needed for all $\theta \in [0,2\pi]$ and
\begin{equation}
\left\{
\begin{array}{l}
\min_{\theta \in [0,2\pi]}  F(\theta)^2=0,\\
\max_{\theta \in [0,2\pi]}  F(\theta)^2= F_x^2+ F_y^2,
\end{array}
\right.
\label{2minmax}
\end{equation}
one can deduce the theorem \ref{2hypcritthm2d}.
\end{proof}

\noindent
{\em Remark:} The hyperbolicity of the two-layer shallow water model with free surface is very different depending on the dimension considered. Moreover, it is clear the set $|Ê F_x | >  F ^{+}_{crit}$ is not a physical one, as it is well-known that, under a strong shear of velocity, {\em Kelvin-Helmholtz} instabilities will arise and will generate mixing between layers: the assumptions of the model will no more be valid.

\noindent
To conclude, considering $\gamma \in ]0,1[$, the exact set of hyperbolicity, $\mathcal{H}_{\gamma}$, is defined by
\begin{equation}
{\bf u} \in \mathcal{H}_{\gamma} \Longleftrightarrow
\left\{
\begin{array}{l}
\inf_{X \in \mathbb{R}^2} h_1(X) > 0,\ \inf_{X \in \mathbb{R}^2} h_2(X) > 0,\\
\inf_{X \in \mathbb{R}^2}  F ^{-}_{crit}(X)^2 -  F_x(X)^2-  F_y(X)^2 > 0.
\end{array}
\right.
\label{2Hgamma}
\end{equation}

\section{Hyperbolicity in the region $\boldsymbol{0 < 1-\gamma \ll 1}$}
In this section, in order to compare $\mathcal{S}^s_{\gamma}$ and $\mathcal{H}_{\gamma} \cap \mathcal{H}^s(\mathbb{R}^2)^6$, asymptotic expansions of $ F ^{\pm}_{crit}$ is performed. Then, to prove a weaker criterion of local well-posedness, in $\mathcal{H}^s(\mathbb{R}^2)$, than (\ref{2condwellposed}), expansion of $\sigma \left(\mathsf{A}({\bf u},\theta) \right)$ is carried out and diagonalizability of $\mathsf{A}({\bf u},\theta)$ is proved, under weak density-stratification.

\subsection{Expansion of ${\bf  F ^{\pm}_{crit}}$}
We define the function $f : \mathbb{R} \times \mathbb{R}_+ \times [0,1] \rightarrow \mathbb{R}$ such that $f(z,h,\gamma):=q(z)$. The next proposition compares $\mathcal{S}^s_{\gamma}$ and $\mathcal{H}_{\gamma} \cap \mathcal{H}^s(\mathbb{R}^2)^6$, under weak density-stratification.

\begin{proposition}
Let $\gamma \in ]0,1[$ such that $1-\gamma$ is small. Then
\begin{equation}
\mathcal{S}^s_{\gamma} \subset \mathcal{H}_{\gamma} \cap \mathcal{H}^s(\mathbb{R}^2)^6
\label{2HgOg}
\end{equation}
\label{2lemHgOg}
\end{proposition}
\begin{proof}
Around the state $z=0$ and $\gamma=1$, it is obvious that $f(0,h,1)=0$. The $1^{\mathrm{st}}$ order {\em Taylor} expansion of $f(z,h,\gamma)$ about this state is
\begin{equation}
\begin{array}{rl}
f(z,h,\gamma)= & f(0,h,1) + z \frac{\partial f}{\partial z}(0,h,1)\\
&  + (\gamma-1)  \frac{\partial f}{\partial \gamma}(0,h,1)+ o\left(z,\gamma-1\right).
\label{2firstTaylor}
\end{array}
\end{equation}
\begin{lemma}
Let $h \in \mathbb{R}_+$, then
\begin{equation}
\frac{\partial f}{\partial z}(0,h,1)=-(1+h)^3,\ \frac{\partial f}{\partial \gamma}(0,h,1)=-(1+h)^4,
\label{2derivdeltafm}
\end{equation}
\label{2lemderivdeltafm}
\end{lemma}

\noindent
consequently, an expansion of $ F ^{-}_{crit}$ is
\begin{equation}
\begin{array}{ll}
 F ^{-\ 2}_{crit}&=(1-\gamma)(1+h)+\mathcal{O}((1-\gamma)^2)\\
 &=(1-\gamma)h+(1-\gamma)(1+o(1)).
\end{array}
\label{2deltaFcritm}
\end{equation}

\noindent
Then it is clear that for all $X \in \mathbb{R}^2$, $F ^{-\ 2}_{crit}(X) > (1-\gamma)h(X)$. Therefore, if ${\bf u} \in \mathcal{S}^s_{\gamma}$, it verifies conditions (\ref{2O}), which imply conditions (\ref{2Hgamma}) and ${\bf u} \in \mathcal{H}_{\gamma}$.
\end{proof}

\noindent
Moreover, another interesting comparison is between the rigid lid model (see \cite{long1956long}) and the free surface one. The exact set of hyperbolicity of the $1^{\mathrm{st}}$ one is characterized by
\begin{equation}
F_x^2+F_y^2 < F_{crit}^{rig\ 2},
\label{2hypriglid}
\end{equation}
with $F_{crit}^{rig\ 2}=(1-\gamma)(1+h)$ . This is compatible with the expansion (\ref{2deltaFcritm}) but does not indicate which model gets the largest set of hyperbolicity. In the next proposition, the comparison of these critical quantities is made.
\begin{proposition}
Let $\gamma \in ]0,1[$ such that $1-\gamma$ is small. If the rigid lid model is hyperbolic, then the free surface one is also hyperbolic:
\begin{equation}
F_{crit}^{-\ 2} > F_{crit}^{rig\ 2}.
\label{2Fcritrigidfree}
\end{equation}
\label{2signFrcrit-}
\end{proposition}
\begin{proof}
With a $2^{\mathrm{nd}}$ order {\em Taylor} expansion of $f(z,h,\gamma)$ about the state $z=0$ and $\gamma=1$, one can check that
\begin{equation}
F_{crit}^{-\ 2} - F_{crit}^{rig\ 2} = (1-\gamma)^2 \frac{h\left( 1+27h+27h^2+9h^3 \right)}{(1+h)^4}+\mathcal{O}((1-\gamma)^3),
\label{2ordreFcrie1gamma}
\end{equation}
then, if $1-\gamma$ is sufficiently small, the expansion (\ref{2ordreFcrie1gamma}) is true and we have
\begin{equation}
F_{crit}^{-\ 2} - F_{crit}^{rig\ 2}>0,
\label{2Frcitrigifree}
\end{equation}
and the proposition \ref{2signFrcrit-} is straightforward proved.
\end{proof}

\noindent
Finally, even if the expansion of the quantity $F_{crit}^+$ is not necessary, as we proved in the theorem \ref{2hypcritthm2d}, the hyperbolicty of the two-dimensional model does not depend of $F_{crit}^+$. It is interesting to know the behavior of the hyperbolicity of the one-dimensional model. We perform expansion about the state $\gamma=1$, because the roots of $q(z)=f(z,h,1)$ are explicit.
\begin{equation}
f(z,h,1)=z^4-3z^3(1+h)+3z^2(h^2-7h+1)-z(1+h)^3.
\label{2fgamma1}
\end{equation}
Then, the expansion of $ F ^{+\ 2}_{crit}$ is the only non-zero and real root of $f(z,h,1)$:
\begin{equation}
 F_{crit}^+=\left[1+h^{\frac{1}{3}} \right]^{\frac{3}{2}}+O(1-\gamma),
\label{2deltaFcritp}
\end{equation}

\noindent
The expansions of $F_{crit}^{\pm}$ are similar to \cite{liska1997analysis} in the case $\gamma=1$ and \cite{stewart2012multilayer} in the general case.

\subsection{Expansion of the spectrum of $\boldsymbol{\mathsf{A}(u,\theta)}$}
In this subsection, eigenvalues of $\mathsf{A}(u,\theta)$ are expanded, in order to prove the diagonalizability of this matrix in the next subsection. Using the rotational invariance, eigenvalues of $\mathsf{A}({\bf u},\theta)$ are deduced from $\sigma\left(\mathsf{A}_x({\bf u})\right)$. The spectrum $\mathsf{A}_x({\bf u})$ is set as
\begin{equation}
\sigma\left(\mathsf{A}_x({\bf u})\right):=\{\mu^{\pm}_1,\mu^{\pm}_2,\mu_3^{\pm}\}.
\label{2spectrumA}
\end{equation}
We define the undimensionned quantities:
\begin{equation}
\forall i\in [\![1,3]\!], \lambda_i^{\pm}:=\frac{\mu_i^{\pm}-u_1}{\sqrt{g h_1}}.
\label{2lambda}
\end{equation}
Then, $\mu_i^{\pm}$ is an eigenvalue of $\mathsf{A}_x({\bf u})$ if and only if $g(\lambda_i^{\pm},F_x,h,\gamma)=0$, where $g: \mathbb{R}^3 \times [0,1] \rightarrow \mathbb{R}$ is defined by
\begin{equation}
g(\lambda,F_x,h,\gamma):=P(\lambda)= \left[\lambda^2 - 1 \right] \left[ \left(\lambda -  F_x \right)^2 - h \right] - \gamma h.
\label{2defg}
\end{equation}
As we get the exact criterion of hyperbolicity (\ref{2Hgamma}), the main goal is to know the conditions to have $\mathsf{A}({\bf u},\theta)$ diagonalizable, not only to get an eigenbasis of $\mathbb{R}^6$ to provide the {\em Riemann} invariants but also to prove that the model is locally well-posed in $\mathcal{H}^s(\mathbb{R}^2)$ (see proposition \ref{2diagimpliqsym}).

\noindent
In the next paragraphs, as there are two trivial eigenvalues: $u_1$ and $u_2$, we settle down $\mu_3^-=u_1$ and $\mu_3^+=u_2$ and asymptotic expansions are performed on $\mu_1^{\pm}$ and $\mu_2^{\pm}$ ({\em i.e.} $\lambda_1^{\pm}$ and $\lambda_2^{\pm}$).

\subsubsection{${\bf |F_x | > F_{crit}^+}$}
As we know $\{\lambda_1^{\pm},\lambda_2^{\pm}\}$ in the case $\gamma=1$ and $h=0$
\begin{equation}
\lambda_1^{\pm}=F_x,\ \lambda_2^{\pm}=\pm 1,
\label{2spectrAgamma11}
\end{equation}
we expand $\{\lambda_1^{\pm},\lambda_2^{\pm}\}$ under assumptions $\gamma \rightarrow 1^-$ and $h \rightarrow 0$. Therefore, as $\lambda_2^{-}$ and $\lambda_2^{+}$ are two distinct eigenvalues, the purpose of this subsection is to know the behavior of $\lambda_1^{\pm}$ when $\gamma \rightarrow 1^-$ and $h \rightarrow 0$, which implies $F_{crit}^+ \rightarrow 1$, according to the expansion (\ref{2deltaFcritp}). The main result is summed up in the next proposition
\begin{proposition}
Let $(\gamma,{\bf u}) \in ]0,1[ \times \mathbb{R}^6$ such that $1-\gamma$ and $h$ are small and $|F_x | > F_{crit}^+$. Then, $\sigma \left( \mathsf{A}_x({\bf u}) \right) \subset \mathbb{R}$ and
\begin{equation}
\left\{
\begin{array}{l}
\lambda_1^{\pm}= F_x + \frac{F_x h}{F_x^2-1}\pm h^{\frac{1}{2}}\left[ 1+ \frac{h \gamma}{F_x^2-1}+\left(\frac{h F_x}{F_x^2-1}\right)^2  \right]^{\frac{1}{2}}+\mathcal{O}(h^{\frac{3}{2}},1-\gamma),\\
\lambda_2^{\pm}=\pm \left[ 1+\frac{h}{2(F_x-1)^2} \right]+\mathcal{O}(h^2,1-\gamma).
\end{array}
\right.
\label{2lambda12pm1}
\end{equation}
\label{2propspectr1}
\end{proposition}
\begin{proof}
The $2^{\mathrm{nd}}$ order {\em Taylor} expansion of $g(\lambda,F_x,h,\gamma)$ about the state $\lambda=F_x$, $h=0$ and $\gamma=1$ provides the expansion of $\lambda_1^{\pm}$: $g(\lambda,F_x,h,\gamma)$ is equal to
\begin{equation}
\begin{array}{ll}
& g(F_x,F_x,0,1)\\
& + (\lambda-F_x) \frac{\partial g}{\partial \lambda}(F_x,F_x,0,1)+ h \frac{\partial g}{\partial h}(F_x,F_x,0,1)+ (\gamma-1) \frac{\partial g}{\partial \gamma}(F_x,F_x,0,1)\\
&  \frac{1}{2}\left[ (\lambda-F_x)^2 \frac{\partial^2 g}{\partial \lambda^2}(F_x,F_x,0,1)+h^2 \frac{\partial^2 g}{\partial h^2}(F_x,F_x,0,1)+(\gamma-1)^2 \frac{\partial^2 g}{\partial \gamma^2}(F_x,F_x,0,1) \right]\\
&+ (\lambda-F_x)h \frac{\partial^2 g}{\partial \lambda \partial h}(F_x,F_x,0,1)+ (\lambda-F_x) (\gamma-1) \frac{\partial^2 g}{\partial \lambda \partial \gamma}(F_x,F_x,0,1)\\
& +  h (\gamma-1) \frac{\partial^2 g}{\partial h \partial \gamma}(F_x,F_x,0,1)+o\left((\lambda-F_x)^2,h^2,(\gamma-1)^2\right)
\label{2secTaylor1}
\end{array}
\end{equation}

\begin{lemma}
$\forall F_x \in \mathbb{R}$,
\begin{equation}
\left\{
\begin{array}{l}
g(F_x,F_x,0,1)=0\\
\frac{\partial g}{\partial \lambda}(F_x,F_x,0,1)=0,\\
\frac{\partial g}{\partial \gamma}(F_x,F_x,0,1)=0,\\
\frac{\partial^2 g}{\partial h^2}(F_x,F_x,0,1)=0,\\
\frac{\partial^2 g}{\partial \gamma^2}(F_x,F_x,0,1)=0,\\
\frac{\partial^2 g}{\partial \lambda \gamma}(F_x,F_x,0,1)=0,\\
\end{array}
\right.
\left\{
\begin{array}{l}
\frac{\partial g}{\partial h}(F_x,F_x,0,1)=-F_x^2,\\
\frac{\partial^2 g}{\partial \lambda^2}(F_x,F_x,0,1)=2(F_x^2-1),\\
\frac{\partial^2 g}{\partial \lambda \partial h}(F_x,F_x,0,1)=-2 F_x,\\
\frac{\partial^2 g}{\partial h \gamma}(F_x,F_x,0,1)=-1,\\
\end{array}
\right.
\label{2derivdeltagpm}
\end{equation}
\label{2lemderivdeltagp}
\end{lemma}

\noindent
Consequently, using the implicit functions theorem, the expansion of $\lambda_1^{\pm}$ is deduced. Moreover, $\lambda_1^{\pm}$ is real if $|F_x| > F_{crit}^+$, because $F_{crit}^+>1$ (see (\ref{2deltaFcritp})). Moreover, with $1^{\mathrm{st}}$ order {\em Taylor} expansion of $g(\lambda,F_x,h,\gamma)$ about $(\lambda,F_x,h,\gamma)=(\pm1,F_x,0,1)$ and implicit theorem, one can get the expansion of $\lambda_2^{\pm}$. Moreover, $\lambda_2^{\pm}$ is unconditionally real.
\end{proof}

\noindent
{\em Remarks:} 1) As it was mentioned before, the expansion of $\lambda_2^{\pm}$ is not necessary to prove the diagonalizability of $\mathsf{A}_x({\bf u})$, but it is interesting to get a more precise expression. 2) We could perform an analyis more general than the one about the state $(h,\gamma)=(0,1)$ as all the calculs are explicit but it is much simpler in this particular case.

\subsubsection{${\bf | F_x| < F_{crit}^-}$}
According to the expansion (\ref{2deltaFcritm}), $\gamma=1$ implies $F_{crit}^{-}=0$. Then, under the assumption $1-\gamma$ small, $| F_x| < F_{crit}^-$ is equivalent to $F_x=0$. As we know exactly $\{\lambda_1^{\pm},\lambda_2^{\pm}\}$ in the particular case $\gamma=1$ and $F_x=0$
\begin{equation}
\lambda_1^{\pm}=\pm \sqrt{1+h},\ \lambda_2^{\pm}=0,
\label{2spectrAgamma1}
\end{equation}
we expand $\{\lambda_1^{\pm},\lambda_2^{\pm}\}$ under assumption $\gamma \rightarrow 1^-$ and $| F_x | < F_{crit}^-$
Therefore, $\lambda_1^{\pm}$ give two distinct eigenvalues, so the main purpose of this subsection is to know the behavior of $\lambda_2^{\pm}$ when and $\gamma \rightarrow 1^-$, which implies $F_{crit}^- \rightarrow 0$ according to (\ref{2deltaFcritm}), as it was noticed above, and consequently $|F_x| \rightarrow 0$.

\begin{proposition}
Let $(\gamma,{\bf u}) \in ]0,1[ \times \mathbb{R}^6$ such that $1-\gamma$ is small and $|F_x | < F_{crit}^-$. Then, $\sigma \left( \mathsf{A}_x({\bf u}) \right) \subset \mathbb{R}$ and
\begin{equation}
\left\{
\begin{array}{l}
\lambda_1^{\pm}=\frac{1}{(1+h)^{\frac{3}{2}}}\left[ F_x h (1+h)^{\frac{1}{2}}\pm \left( (1+h)^2-\frac{1}{2}h(1-\gamma) \right) \right]+\mathcal{O}(1-\gamma),\\
\lambda_2^{\pm}= \frac{F_x}{1+h}\pm \left[\frac{h}{(1+h)^2}\left( (1+h)(1-\gamma) -F_x^2 \right)\right]^{\frac{1}{2}}+\mathcal{O}(1-\gamma).
\end{array}
\right.
\label{2lambda12pm2}
\end{equation}
\label{2propspectr2}
\end{proposition}
\begin{proof} The $2^{\mathrm{nd}}$ order {\em Taylor} expansion of $g(\lambda,F_x,h,\gamma)$ about the state $\lambda=0$, $F_x=0$ and $\gamma=1$ provides the expand of $\lambda_2^{\pm}$: $g(\lambda,F_x,h,\gamma)$ is equal to
\begin{equation}
\begin{array}{l}
g(0,0,h,1)\\
 + \lambda \frac{\partial g}{\partial \lambda}(0,0,h,1)+ F_x \frac{\partial g}{\partial F_x}(0,0,h,1)+ (\gamma-1) \frac{\partial g}{\partial \gamma}(0,0,h,1)\\
  \frac{1}{2}\left[ \lambda^2 \frac{\partial^2 g}{\partial \lambda^2}(0,0,h,1)+F_x^2 \frac{\partial^2 g}{\partial F_x^2}(0,0,h,1)+(\gamma-1)^2 \frac{\partial^2 g}{\partial \gamma^2}(0,0,h,1) \right]\\
+ \lambda F_x \frac{\partial^2 g}{\partial \lambda \partial F_x}(0,0,h,1)+ \lambda (\gamma-1) \frac{\partial^2 g}{\partial \lambda \partial \gamma}(0,0,h,1)\\
 +  F_x (\gamma-1) \frac{\partial^2 g}{\partial F_x \partial \gamma}(0,0,h,1)+o\left(\lambda^2,F_x^2,(\gamma-1)^2\right)
\label{2secTaylor}
\end{array}
\end{equation}
\begin{lemma}
$\forall h \in \mathbb{R}_+^*$,
\begin{equation}
\left\{
\begin{array}{l}
g(0,0,h,1)=0\\
\frac{\partial g}{\partial \lambda}(0,0,h,1)=0,\\
\frac{\partial g}{\partial F_x}(0,0,h,1)=0,\\
\frac{\partial^2 g}{\partial \gamma^2}(0,0,h,1)=0,\\
\frac{\partial^2 g}{\partial \lambda \partial \gamma}(0,0,h,1)=0,\\
\frac{\partial^2 g}{\partial F_x \partial \gamma}(0,0,h,1)=0,\\
\end{array}
\right.
\left\{
\begin{array}{l}
\frac{\partial g}{\partial \gamma}(0,0,h,1)=-h,\\
\frac{\partial^2 g}{\partial \lambda^2}(0,0,h,1)=-2(1+h),\\
\frac{\partial^2 g}{\partial F_x^2}(0,0,h,1)=-2,\\
\frac{\partial^2 g}{\partial \partial \lambda F_x}(0,0,h,1)=2.\\
\end{array}
\right.
\label{2derivdeltagm}
\end{equation}
\label{2lemderivdeltagm}
\end{lemma}

\noindent
Consequently, if $|F_x| < F_{crit}^-$,  $\lambda_2^{\pm}$ is real and using the implicit functions theorem, the expansion of $\lambda_2^{\pm}$ is insured. Moreover, with $1^{\mathrm{st}}$ order {\em Taylor} expansion of $g(\lambda,F_x,h,\gamma)$ about $(\lambda,F_x,h,\gamma)=(\pm \sqrt{1+h},0,h,1)$ and the implicit theorem, one can get the expansion of $\lambda_1^{\pm}$. Moreover, $\lambda_1^{\pm}$ is unconditionally real.
\end{proof}

\noindent
In this set of hyperbolicity ({\em i.e.} $|F_x| < F_{crit}^-$), expansion are in accordance with \cite{schijf1953theoretical}, \cite{ovsyannikov1979two}, \cite{kim2008two}, \cite{abgrall2009two} and \cite{stewart2012multilayer}.

\noindent
{\em Remarks:} 1) Approximations (\ref{2lambda12pm2}) are precise in $\mathcal{O}(1-\gamma)$, and not $\mathcal{O}(F_x^2,(1-\gamma))$, because if $F_x^2 < F_{crit}^{-\ 2}$ then $F_x^2=\mathcal{O}(1-\gamma)$, according to expansion (\ref{2deltaFcritm}). 2) The expansion of $\lambda_1^{\pm}$ is not necessary, but it is interesting to get a more precise expression.

\subsection{The eigenstructure of $\boldsymbol{\mathsf{A}({\bf u},\theta)}$}
The description of the eigenstructure is a decisive point, as it permits to caracterize exactly the {\em Riemann} invariants and the local well-posedness in $\mathcal{H}^s(\mathbb{R}^2)$ (see proposition \ref{2diagimpliqsym}).

\begin{proposition}
There exists $\delta >0$ such that if $\gamma \in ]1-\delta,1[$ and $({\bf u},\theta) \in \mathcal{H}_{\gamma} \times [0,2\pi]$, the matrix $\mathsf{A}({\bf u},\theta)$ is diagonalizable.
\label{2diagA}
\end{proposition}
\begin{proof}
With the rotational invariance (\ref{2rotinv}), it is equivalent to prove the diagonalizability of $\mathsf{A}_x({\bf u})$. By denoting $({\bf e_i})_{i \in [\![1,6]\!]}$ the canonical basis of $\mathbb{R}^6$, one can prove the right eigenvectors $\boldsymbol{r}_x^{\mu}({\bf u})$ of $\mathsf{A}_x({\bf u})$, associated to the eigenvalue $\mu$, are defined by
\begin{equation}
\left\{
\begin{array}{ll}
\boldsymbol{r}_x^{\mu}({\bf u})={\bf e_1}+\frac{\mu-u_1}{h_1}{\bf e_3}-c_{\mu}^r\left({\bf e_2}+\frac{ \mu-u_2}{h_2}{\bf e_4}\right),& \mathrm{if}\ \mu \in \{\mu_1^{\pm},\mu_2^{\pm}\},\\
\boldsymbol{r}_x^{\mu}({\bf u})={\bf e_5},&\mathrm{if}\ \mu=\mu_3^{-},\\
\boldsymbol{r}_x^{\mu}({\bf u})={\bf e_6},& \mathrm{if}\ \mu=\mu_3^{+},\\
\end{array}
\right.
\label{2eigvectx}
\end{equation}
where $c_{\mu}^r:=1-\frac{(\mu-u_1)^2}{g h_1}$. Then, the right eigenvectors $\boldsymbol{r}^{\mu}({\bf u},\theta)$ of $\mathsf{A}({\bf u},\theta)$ are defined by
\begin{equation}
\forall \mu \in \sigma\left(\mathsf{A}({\bf u},\theta) \right),\ \boldsymbol{r}^{\mu}({\bf u},\theta)=\mathsf{P}(\theta)^{-1}\boldsymbol{r}_x^{\mu}(\mathsf{P}(\theta){\bf u}).
\label{2eigvect}
\end{equation}

\noindent
Moreover, if $1-\gamma$ is sufficiently small ({\em i.e.} $\gamma \in ]1-\delta,1[$), the eigenvalues $\mu_i^{\pm}$, with $i \in \{1,2\}$, are all distinct (the existence of $\delta>0$ is guaranteed). Indeed, there is the next inequalities if $\gamma \in ]1-\delta,1[$
\begin{equation}
\mu_1^{+} > \mu_2^+ > \mu_2^- > \mu_1^-
\label{2ineqmui12}
\end{equation}
Consequently, as the eigenvalues are real, the right-eigenvectors (\ref{2eigvectx}) constitute an eigenbasis of $\mathbb{R}^6$ and $\mathsf{A}_x({\bf u})$ is diagonalizable.
\end{proof}

\noindent
{\em Remark:} There is also the left eigenvectors $\boldsymbol{l}_x^{\mu}({\bf u})$ of $\mathsf{A}_x({\bf u})$, associated to the eigenvalue $\mu \in \sigma(\mathsf{A}_x({\bf u}))$:
\begin{equation}
\left\{
\begin{array}{ll}
{}^{\top}\boldsymbol{l}_x^{\mu}({\bf u})=\frac{\mu-u_1}{h_1}{\bf e_1}+{\bf e_3}-c_{\mu}^l\left(\frac{\mu-u_2}{h_2}{\bf e_2}-{\bf e_4}\right),& \mathrm{if}\ \mu \in \{\mu_1^{\pm},\mu_2^{\pm}\},\\
{}^{\top}\boldsymbol{l}_x^{\mu}({\bf u})={\bf e_5},&\mathrm{if}\ \mu=\mu_3^{-},\\
{}^{\top}\boldsymbol{l}_x^{\mu}({\bf u})={\bf e_6},& \mathrm{if}\ \mu=\mu_3^{+},\\
\end{array}
\right.
\label{2eigvectxleft}
\end{equation}
where $c_{\mu}^l:=1-\frac{(\mu-u_2)^2}{g h_2}$. And the left eigenvectors $\boldsymbol{l}^{\mu}({\bf u},\theta)$ of $\mathsf{A}({\bf u},\theta)$ are defined by
\begin{equation}
\forall \mu \in \sigma\left(\mathsf{A}({\bf u},\theta) \right),\ \boldsymbol{l}^{\mu}({\bf u},\theta)=\boldsymbol{l}_x^{\mu}(\mathsf{P}(\theta){\bf u})\mathsf{P}(\theta).
\label{2eigvectleft}
\end{equation}

\noindent
Furthermore, in order to know the type of the wave associated to each eigenvalue -- shock, contact or rarefaction wave -- there is the next proposition
\begin{proposition}
There exists $\delta >0$ such that if $\gamma \in ]1-\delta,1[$, then
\begin{equation}
\left\{
\begin{array}{ll}
\mathrm{the}\ \mu_i^{\pm}\mathrm{-characteristic\ field\ is\ genuinely\ non-linear},&\mathrm{if}\ i \in \{1,2\},\\
\mathrm{the}\ \mu_i^{\pm}\mathrm{-characteristic\ field\ is\ linearly\ degenerate},&\mathrm{if}\ i =3.
\end{array}
\right.
\label{2muwavenature}
\end{equation}
\label{2genuinelynonlin}
\end{proposition}
\begin{proof}
If $\gamma$ is sufficiently close to $1$ ({\em i.e.} $\gamma \in ]1-\delta,1[$, with $\delta>0$) the expansions (\ref{2lambda12pm2}), when $|F_x|<F_{crit}^-$, are valid. Moreover, we remark that $\mu_i^{\pm}$ depends analytically of the parameters of the problem and we deduce that the $o(1-\gamma)$ still remains small after derivating. Then, with the expression of the right eigenvectors (\ref{2eigvectx}) of $\mathsf{A}_x({\bf u})$, one can check
\begin{equation}
\left\{
\begin{array}{ll}
\nabla \mu_i^{\pm} \cdot \boldsymbol{r}_x^{\mu}({\bf u}) \not=0,& \mathrm{if}\ i \in \{1,2\},\\
\nabla \mu_i^{\pm} \cdot \boldsymbol{r}_x^{\mu}({\bf u}) =0,& \mathrm{if}\ i=3,
\label{2genuidef}
\end{array}
\right.
\end{equation}
then, the proposition \ref{2genuinelynonlin} is proved.
\end{proof}

\noindent
{\em Remark:} when $u_2-u_1$ and $1-\gamma$ are both equal to $0$, the $\mu_1^{\pm}$-characteristic field remains genuinely non-linear but the $\mu_2^{\pm}$-characteristic field becomes linearly degenerate.

\noindent
To conclude, under conditions of the proposition \ref{2genuinelynonlin}, for all $i \in \{1,2\}$, the $\mu_i^{\pm}$-wave is a shock wave or a rarefaction wave and the $\mu_3^{\pm}$-wave is a contact wave.

\noindent
Finally, as a consequence, we deduce a criterion of local well-posedness in $\mathcal{H}^s(\mathbb{R}^2)$, more general than criterion (\ref{2O}).

\begin{corollary}
There exists $\delta >0$ such that if $\gamma\in ]1-\delta,1[$ and ${\bf u^0} \in \mathcal{H}_{\gamma} \cap \mathcal{H}^s(\mathbb{R}^2)^6$, then, the {\em Cauchy} problem, associated with {\em (\ref{2systemmultilayer})} and initial data ${\bf u^0}$, is locally well-posed in $\mathcal{H}^s(\mathbb{R}^2)$, hyperbolic and the unique solution verifies conditions {\em (\ref{2strongsol})}.
\label{2wellprosedhyp2d}
\end{corollary}
\begin{proof}
Let $\gamma \in ]0,1[$. As it was proved in the proposition \ref{2diagA}, there exists $\delta>0$ such that if $\gamma \in ]1-\delta,1[$ then for all $({\bf u},\theta) \in \mathcal{H}_{\gamma} \times [0,2\pi]$, $\mathsf{A}({\bf u^0},\theta)$ is diagonalizable. Then, by definition of $\mathcal{H}_{\gamma}$, the {\em Cauchy} problem is hyperbolic. Moreover, according to proposition \ref{2diagimpliqsym}, it is locally well-posed in $\mathcal{H}^s(\mathbb{R}^2)$ and the unique solution verifies conditions (\ref{2strongsol}).
\end{proof}

\noindent
{\em Remark:} This criterion is less restrictive than (\ref{2O}), because as it was proved in proposition \ref{2lemHgOg}: if $1-\gamma$ is sufficiently small, $\mathcal{S}^s_{\gamma} \subset \mathcal{H}_{\gamma}$.

\section{A {\em conservative} two-layer shallow water model}
Even if the model (\ref{2massconservation}--\ref{2momentumconservation}) is conservative, in the one-dimensional case, with the unknowns $(h_i,u_i)$, with $i \in \{1,2\}$. It is not anymore true in the two-dimensional case. This subsection will treat this lack of conservativity by an augmented model, with a different approach from \cite{abgrall2009two}. We remind that no assumption has been made concerning the horizontal vorticity, in each layer
\begin{equation}
w_i:=curl(\boldsymbol{ u_i})=\frac{\partial v_i}{\partial x} - \frac{\partial u_i}{\partial y},\ i \in \{1,2\}.
\label{2wi}
\end{equation}

\subsection{Conservation laws}
Using a {\em Frobenius} problem, it was proved in \cite{barros2006conservation} that the one-dimensional two-layer shallow water model with free surface has a finite number of conservative quantities: the height and velocity in each layer, the total momentum and the total energy. However, in the two-dimensional case, it is still an open question. Nevertheless, introducing $w_i$, for $i \in \{1,2\}$, in equations (\ref{2massconservation}--\ref{2momentumconservation}), the conservation of mass (\ref{2massconservation}) is unchanged
\begin{equation}
\frac{\partial h_i}{\partial t} + {\bf \nabla} {\bf \cdot} (h_i \boldsymbol{u_i}) = 0,
\label{2massconservationrelax}
\end{equation}
but the equation of depth-averaged horizontal velocity (\ref{2momentumconservation}) becomes conservative
\begin{equation}
\frac{\partial \boldsymbol{u_i}}{\partial t}+{\bf \nabla} \left(\frac{1}{2} (u_i^2+v_i^2) +P_i \right) - (f+w_i) \boldsymbol{u_i}^{\bot} = 0.
\label{2momentumconservationrelax}
\end{equation}
Moreover, the horizontal vorticity in each layer is also conservative
\begin{equation}
\frac{\partial w_i}{\partial t} + \nabla \cdot \left((w_i+f)\boldsymbol{u_i} \right)=0.
\label{2rot}
\end{equation}
Therefore, in the two-dimensional case, there are at least $8$ conservative quantities: the height, the velocity and the horizontal vorticity in each layer, the total momentum and the energy $e_2$:
\begin{equation}
e_2:=\frac{1}{2}\gamma h_1 \left(u_1^2+v_1^2+  g (h_1+2 h_2)\right)+\frac{1}{2} h_2 \left( u_2^2+v_2^2+ g h_2\right).
\label{2energy2d}
\end{equation}

\subsection{A new augmented model}
From equations (\ref{2massconservation}--\ref{2momentumconservation}), it is possible to obtain a new model. We denote $({\bf u},{\bf v}) \in \mathcal{H}^s(\mathbb{R}^2)^6 \times \mathcal{H}^s(\mathbb{R}^2)^8$, the vectors defined by
\begin{equation}
\left\{
\begin{array}{l}
{\bf u}:={}^{\top} (h_1,h_2,u_1,u_2,v_1,v_2),\\
{\bf v}:={}^{\top} (h_1,h_2,u_1,u_2,v_1,v_2,w_1,w_2).
\end{array}
\right.
\label{2uvdefinition}
\end{equation}
If ${\bf u}$ is a classical solution of (\ref{2systemmultilayer}), then ${\bf v}$ is solution of the augmented system
\begin{equation}
\frac{\partial {\bf v}}{\partial t} + \mathsf{A^r}_x({\bf v}) \frac{\partial {\bf v}}{\partial x} + \mathsf{A^r}_y ({\bf v}) \frac{\partial {\bf v}}{\partial y} + {\bf b^r}({\bf v})=0,
\label{2systemmultilayerrelax}
\end{equation}
where  $\mathsf{A^r_x}({\bf v})$, $ \mathsf{A^r_y}({\bf v})$ and ${\bf b^r}({\bf v})$ are defined by
\begin{equation}
 \mathsf{A^r}_x({\bf v}) :=  \left[
\begin{array}{cccccccc}
 u_1 & 0 & h_1 & 0 & 0 & 0 & 0 & 0\\
 0 & u_2 & 0 & h_2 & 0 & 0 & 0 & 0 \\
 g & g & u_1 & 0 & v_1 & 0 & 0 & 0  \\
 \gamma g & g & 0 & u_2 & 0 & v_2 & 0 & 0 \\
 0 & 0 & 0 & 0 & 0 & 0 & 0 & 0 \\
 0 & 0 & 0 & 0 & 0 & 0 & 0 & 0 \\
 0 & 0 & w_1+f & 0 & 0 & 0 & u_1 & 0 \\
 0 & 0 & 0 & w_2+f & 0 & 0 & 0 & u_2 \\
\end{array}
\right] ,
\label{2Arx}
\end{equation}

\begin{equation}
\mathsf{A^r}_y({\bf v}) :=  \left[
\begin{array}{cccccccc}
 v_1 & 0 & 0 & 0 & h_1 & 0 & 0 & 0 \\
 0 & v_2 & 0 & 0 & 0 & h_2 & 0 & 0 \\
 0 & 0 & 0 & 0 & 0 & 0 & 0 & 0 \\
 0 & 0 & 0 & 0 & 0 & 0 & 0 & 0 \\
 g & g & u_1 & 0 & v_1 & 0 & 0 & 0  \\
 \gamma g & g & 0 & u_2 & 0 & v_2 & 0 & 0 \\
  0 & 0 & 0 & 0 & w_1+f & 0 & 0 & 0 \\
 0 & 0 & 0 & 0 & 0 & w_2+f & 0 & 0 \\
\end{array}
\right],
\label{2Ary}
\end{equation} 
\begin{equation}
\begin{array}{ll}
{\bf b^r}({\bf v}):= & {}^{\top} \left(0,0,-(w_1+f) v_1,-(w_2+f) v_2,(w_1+f) u_1,(w_2+f) u_2\right)\\
& + {}^{\top}\left(0,0,g \frac{\partial b}{\partial x},g \frac{\partial b}{\partial x},g \frac{\partial b}{\partial x},g \frac{\partial b}{\partial x}\right).
\end{array}
\label{2br}
\end{equation}
Even if the model (\ref{2massconservation}--\ref{2momentumconservation}) is not conservative, the model (\ref{2systemmultilayerrelax}) is always conservative. Then, there is no need to chose a conservative path in the numerical resolution.
\noindent
{\em Remark:} $e_2$ is not an energy of the augmented model (\ref{2systemmultilayerrelax}). Indeed, it is never a convex function with the variable ${\bf v}$ as it is independent of $w_1$ and $w_2$.

\begin{proposition}
The augmented model {\em (\ref{2systemmultilayerrelax})} verifies the rotational invariance.
\label{2proprotinvrelax}
\end{proposition}

\begin{proof}
We denote by $\mathsf{A^r}({\bf v},\theta)$ the matrix defined by $\cos(\theta) \mathsf{A_x^r}({\bf v})+\sin(\theta) \mathsf{A^r}_y({\bf v})$. One can check the next equality, for all $({\bf v},\theta) \in \mathbb{R}^8 \times [0,2\pi]$
\begin{equation}
\mathsf{A^r}({\bf v},\theta)=\mathsf{P^{r}}(\theta)^{-1} \mathsf{A}_x(\mathsf{P^r}(\theta) {\bf v}) \mathsf{P^r}(\theta),
\label{2rotinvrelax}
\end{equation}
where $\mathsf{P^r}(\theta)$ is the $8 \times 8$ matrix defined by
\begin{equation}
\mathsf{P^r}(\theta) :=  \left[
\begin{array}{cccccccc}
 1 & 0 & 0 & 0 & 0 & 0 & 0 & 0 \\
 0 & 1 & 0 & 0 & 0 & 0 & 0 & 0  \\
 0 & 0 & \cos(\theta) & 0 & \sin(\theta) & 0 & 0 & 0  \\
 0 & 0 & 0 & \cos(\theta) & 0 & \sin(\theta) & 0 & 0  \\
 0 & 0 & -\sin(\theta) & 0 & \cos(\theta) & 0 & 0 & 0   \\
 0 & 0 & 0 & -\sin(\theta) & 0 & \cos(\theta) & 0 & 0  \\
 0 & 0 & 0 & 0 & 0 & 0 & 1 & 0  \\
 0 & 0 & 0 & 0 & 0 & 0 & 0 & 1  \\
\end{array}
\right],
\label{2Prtheta}
\end{equation}
and, moreover, we notice $\mathsf{P^r}(\theta)^{-1}={}^{\top}\mathsf{P^r}(\theta)$.
\end{proof}

\subsection{The eigenstructure of $\boldsymbol{\mathsf{A^r}({\bf v},\theta)}$}
As it was reminded before, the description of the eigenstructure of $\mathsf{A^r}({\bf v},\theta)$ is a decisive point, as it permits to caracterize exactly its diagonalizability and also the {\em Riemann} invariants. According to the rotational invariance (\ref{2rotinvrelax}), we restrict the analysis to the eigenstructure of $\mathsf{A}^{\mathsf{r}}_x({\bf v})$. First of all, we define the spectrum of $\mathsf{A_x^r}({\bf v})$ by
\begin{equation}
\sigma(\mathsf{A}^{\mathsf{r}}_x({\bf v})):=\{\nu_i^{\pm},\ i \in [\![1,4]\!]\}.
\label{2spectrumAr}
\end{equation}
As the characteristic polynomial of $\mathsf{A}^{\mathsf{r}}_x({\bf v})$ is equal to
\begin{equation}
\det(\mathsf{A}^{\mathsf{r}}_x({\bf v})-\mu \mathsf{I_8})=\mu^2\det(\mathsf{A}_x({\bf u})-\mu \mathsf{I_6}),
\label{2polcarAr}
\end{equation}
we settle down $\nu_4^{\pm}:=0$ and for all $i \in [\![1,3]\!]$ and $\nu_i^{\pm}:=\mu_i^{\pm}$. Then, we define $\mathcal{H}^r_{\gamma} \subset \mathcal{L}^2(\mathbb{R}^2)^8$, the open subset of initial conditions, such that the system (\ref{2systemmultilayerrelax}) is hyperbolic:
\begin{equation}
\mathcal{H}^r_{\gamma}:=  \left\{ {\bf v}: \mathbb{R}^2 \mapsto \mathbb{R}^6 / {\bf u} \in \mathcal{H}_{\gamma} \right\}.
\label{2Hgammarelax}
\end{equation}

\begin{proposition}
There exists $\delta >0$ such that if $\gamma \in ]1-\delta,1[$ and $({\bf v},\theta) \in \mathcal{H}^r_{\gamma} \times [0,2\pi]$. Then, the matrix $\mathsf{A^r}({\bf v},\theta)$ is diagonalizable.
\label{2diagArelax}
\end{proposition}
\begin{proof}
With the rotational invariance (\ref{2rotinvrelax}), it is equivalent to prove the diagonalizability of $\mathsf{A}^{\mathsf{r}}_x({\bf v})$. By denoting $({\bf e'_i})_{i \in [\![1,8]\!]}$ the canonical basis of $\mathbb{R}^8$, one can prove the right eigenvectors $\boldsymbol{r}_x^{\nu}({\bf u})$ of $\mathsf{A}^{\mathsf{r}}_x({\bf u})$, associated to the eigenvalue $\nu \in \sigma(\mathsf{A}^{\mathsf{r}}_x({\bf u}))$, are defined by
\begin{equation}
\left\{
\begin{array}{ll}
\boldsymbol{r}_x^{\nu}({\bf u})=\begin{array}{ll}&{\bf e'_1}+\frac{1}{h_1}((\nu-u_1){\bf e'_3}+(w_1+f){\bf e'_7})\\ &-c_{\nu}^r\left({\bf e'_2}+\frac{1}{h_2}((\nu-u_2){\bf e'_4}+(w_2+f){\bf e'_8})\right),
\end{array}
& \mathrm{if}\ \nu \in \{\nu_1^{\pm},\nu_2^{\pm}\},\\[4pt]
\boldsymbol{r}_x^{\nu}({\bf u})={\bf e'_7},&\mathrm{if}\ \nu=\nu_3^{-},\\[4pt]
\boldsymbol{r}_x^{\nu}({\bf u})={\bf e'_8},& \mathrm{if}\ \nu=\nu_3^{+},\\[4pt]
\boldsymbol{r}_x^{\nu}({\bf u})=\begin{array}{ll}&v_1 v_2 ( h_2 {\bf e'_2}-u_2 {\bf e'_4}+(f+w_2){\bf e'_8})\\ &-gh_2 v_2 {\bf e'_5}+v_1(u_2^2-g h_2) {\bf e'_6},
\end{array}
&\mathrm{if}\ \nu=\nu_4^{-},\\[4pt]
\boldsymbol{r}_x^{\nu}({\bf u})=\begin{array}{ll}&v_1 v_2 ( h_1 {\bf e'_1}-u_1 {\bf e'_3}+(f+w_1){\bf e'_7})\\ &+v_2(u_1^2-g h_1) {\bf e'_5}-\gamma gh_1 v_1 {\bf e'_6},
\end{array}
& \mathrm{if}\ \nu=\nu_4^{+},\\[4pt]
\end{array}
\right.
\label{2eigvectxrelax}
\end{equation}
where $c_{\nu}^r:=1-\frac{(\nu-u_1)^2}{g h_1}$. Then, the right eigenvectors $\boldsymbol{r}^{\nu}({\bf v},\theta)$ of $\mathsf{A^r}({\bf v},\theta)$ are defined by
\begin{equation}
\forall \nu \in \sigma\left(\mathsf{A^r}({\bf v},\theta) \right),\ \boldsymbol{r}^{\nu}({\bf v},\theta)=\mathsf{P^r}(\theta)^{-1}\boldsymbol{r}_x^{\nu}(\mathsf{P^r}(\theta){\bf v}).
\label{2eigvectrelax}
\end{equation}

\noindent
Moreover, according to the definition of $\delta$, in inequalities (\ref{2ineqmui12}), if $\gamma \in ]1-\delta,1[$, the eigenvalues $\nu_i^{\pm}:=\mu_i^{\pm}$, with $i \in \{1,2\}$, are all distinct. A consequence is the right eigenvectors form an eigenbasis of $\mathbb{R}^8$ and $\mathsf{A}^{\mathsf{r}}_x({\bf v})$ is diagonalizable.
\end{proof}

\noindent
{\em Remark:} There is also the left eigenvectors $\boldsymbol{l}_x^{\nu}({\bf v})$ of $\mathsf{A}^{\mathsf{r}}_x({\bf v})$, associated to the eigenvalue $\nu \in \sigma(\mathsf{A}^{\mathsf{r}}_x({\bf v}))$:
\begin{equation}
\left\{
\begin{array}{ll}
{}^{\top}\boldsymbol{l}_x^{\nu}({\bf v})=\begin{array}{ll}&\nu \frac{\nu-u_1}{h_1}{\bf e'_1}+\nu {\bf e'_3}+v_1 {\bf e'_6}\\ &-c_{\nu}^l\left(\nu\frac{\nu-u_2}{h_2}{\bf e'_2}-\nu{\bf e'_4}-v_2{\bf e'_7}\right),
\end{array}
& \mathrm{if}\ \nu \in \{\nu_1^{\pm},\nu_2^{\pm}\},\\
{}^{\top}\boldsymbol{l}_x^{\nu}({\bf v})={\bf e'_5},&\mathrm{if}\ \nu=\nu_4^{-},\\
{}^{\top}\boldsymbol{l}_x^{\nu}({\bf v})={\bf e'_6},& \mathrm{if}\ \nu=\nu_4^{+},\\
{}^{\top}\boldsymbol{l}_x^{\nu}({\bf v})=-(f+w_1){\bf e'_1}+h_1{\bf e'_7},&\mathrm{if}\ \nu=\nu_3^{-},\\
{}^{\top}\boldsymbol{l}_x^{\nu}({\bf v})=-(f+w_2){\bf e'_2}+h_2{\bf e'_8},& \mathrm{if}\ \nu=\nu_3^{+},\\
\end{array}
\right.
\label{2eigvectxleftrelax}
\end{equation}
where $c_{\nu}^l:=1-\frac{(\nu-u_2)^2}{g h_2}$. And the left eigenvectors $\boldsymbol{l}^{\nu}({\bf v},\theta)$ of $\mathsf{A^r}({\bf v},\theta)$ are also defined by
\begin{equation}
\forall \nu \in \sigma\left(\mathsf{A^r}({\bf v},\theta) \right),\ \boldsymbol{l}^{\nu}({\bf v},\theta)=\boldsymbol{l}_x^{\nu}(\mathsf{P^r}(\theta){\bf v})\mathsf{P^r}(\theta).
\label{2eigvectleftrelax}
\end{equation}

\noindent
Furthermore, the type of the wave associated to each eigenvalue is described in the next proposition.
\begin{proposition}
There exists $\delta>0$ such that if $\gamma \in ]1-\delta,1[$, then
\begin{equation}
\left\{
\begin{array}{ll}
\mathrm{the}\ \nu_i^{\pm}\mathrm{-characteristic\ field\ is\ genuinely\ non-linear},&\mathrm{if}\ i \in \{1,2\},\\
\mathrm{the}\ \nu_i^{\pm}\mathrm{-characteristic\ field\ is\ linearly\ degenerate},&\mathrm{if}\ i \in \{3,4\}.
\end{array}
\right.
\label{2muwavenaturerelax}
\end{equation}
\label{2genuinelynonlinrelax}
\end{proposition}
\begin{proof}
Using the same proof of proposition \ref{2genuinelynonlin} and remarking that $\nu_4^{\pm}=0$, it implies
\begin{equation}
\nabla \nu_4^{\pm} \cdot \boldsymbol{r}_x^{\nu}({\bf v}) =0,
\label{2genuidefrelax}
\end{equation}
and the proposition \ref{2genuinelynonlinrelax} is proved.
\end{proof}

\noindent
To conclude, under conditions of the proposition (\ref{2genuinelynonlinrelax}), for all $i \in \{1,2\}$, the $\nu_i^{\pm}$-wave is a shock wave or a rarefaction wave and for all $i \in \{3,4\}$, the $\nu_i^{\pm}$-wave is a contact wave.

\noindent
Finally, the point is to know if this augmented system (\ref{2systemmultilayerrelax}) is locally well-posed and if its solution provides the solution of the non-augmented system (\ref{2systemmultilayer}).
\begin{theorem}
There exists $\delta>0$ such that if $\gamma \in ]1-\delta,1[$, ${\bf v^0} \in \mathcal{H}^r_{\gamma} \cap \mathcal{H}^s(\mathbb{R}^2)^8$. Then the {\em Cauchy} problem associated with system {\em (\ref{2systemmultilayerrelax})} and initial data ${\bf v^0}$, is locally well-posed in $\mathcal{H}^s(\mathbb{R}^2)$ and hyperbolic: there exists $T>0$ such that ${\bf v}$, the unique solution of the {\em Cauchy} problem, verifies
\begin{equation}
\left\{
\begin{array}{l}
{\bf v} \in \mathcal{C}^1([0,T] \times \mathbb{R}^2)^8,\\
{\bf v} \in \mathcal{C}([0,T];\mathcal{H}^s(\mathbb{R}^2))^8 \cap \mathcal{C}^1([0,T];\mathcal{H}^{s-1}(\mathbb{R}^2))^8.
\end{array}
\right.
\label{2strongsolrelax}
\end{equation}
Furthermore, ${\bf u}$ verifies conditions {\em (\ref{2strongsol})} and is the unique classical solution of the Cauchy problem, associated with {\em (\ref{2systemmultilayer})} and initial data ${\bf u^0}$, if and only if
\begin{equation}
\forall i \in \{1,2\},\ w_i^0 = \frac{\partial v_i^0}{\partial x}-\frac{\partial u_i^0}{\partial y}.
\label{2initialcompatible}
\end{equation}
\label{2thmimplicsystrelax}
\end{theorem}
\begin{proof}
Using proposition \ref{2diagArelax}, $\sigma(\mathsf{A^r}({\bf v},\theta)) \subset \mathbb{R}$ and $\mathsf{A^r}({\bf v},\theta)$ is diagonalizable. Then, the proposition \ref{2diagimpliqsym} is verified: the hyperbolicity and the local well-posedness of the {\em Cauchy} problem, associated with system (\ref{2systemmultilayerrelax}) and initial data ${\bf v^0}$, is insured and conditions (\ref{2strongsolrelax}) are verified. Moreover, it is obvious to prove that, for all $i \in \{1,2\}$, there exists $\phi_i : \mathbb{R}^2 \rightarrow \mathbb{R}$ such that
\begin{equation}
\forall (t,x,y) \in \mathbb{R}_+ \times \mathbb{R}^2,\ w_i(t,x,y)=\frac{\partial v_i}{\partial x}(t,x,y)-\frac{\partial u_i}{\partial y}(t,x,y)+\phi_i(x,y).
\label{2phi}
\end{equation}
As $\phi_i$ does not depend of the time $t$: ${\bf u}$ -- the $6^{\mathrm{th}}$ first coordinates of ${\bf v}$ -- is solution of the non-augmented system (\ref{2systemmultilayer}) if and only if $\phi_i=0$, $\forall i \in \{1,2\}$, which is true if and only if it is verified at $t=0$.
\end{proof}

\section{Discussions and perspectives}
In this article, we proved the hyperbolicity and the local well-posedness, in $\mathcal{H}^s(\mathbb{R}^2)$, of the two-dimensional two-layer shallow water model with free surface with various techniques. All of them use the rotational invariance property (\ref{2rotinv}), reducing the problem from two dimensions to one dimension. We gave, at first, a criterion of local well-posedness, in $\mathcal{H}^s(\mathbb{R}^2)$, using the symmetrizability of the system (\ref{2systemmultilayer}). Afterwards, the exact set of hyperbolicity of this system was explicitly characterized and compared with the set of symmetrizability. Then, after getting an asymptotic expansion of the eigenvalues, we characterized the type of wave associated to each element of the spectrum of $\mathsf{A}_x({\bf u})$ -- shock, rarefaction of contact wave -- and we proved the local well-posedness, in $\mathcal{H}^s(\mathbb{R}^2)$, of the system (\ref{2systemmultilayer}) under conditions of hyperbolicity and weak density-stratification. Finally, we introduced an augmented model (\ref{2systemmultilayerrelax}), adding the horizontal vorticity as a new unknown. We also characterized the type of the waves, proved the local well-posedness in $\mathcal{H}^s(\mathbb{R}^2)$ and explained the link of a solution of the model (\ref{2systemmultilayer}) and a solution of the augmented model (\ref{2systemmultilayerrelax}).

\noindent
As shown in this paper, most of the analysis of the two-dimensional two-layer model with free surface is performed explicitly. In the case of $n$ fluids, with $n \ge 3$, it is not possible anymore. Very few results have been proved concerning the general multi-layer model. Most of them are in particular cases, such as {\em Stewart et al.} \cite{stewart2012multilayer} and \cite{castro2010hyperbolicity} in the three-layer case; \cite{audusse2005multilayer} in the case $\rho_1=\ldots=\rho_n$. In the general case, \cite{duchene2013note} proved the local well-posedness, in one dimension, of the multi-layer model, under conditions of weak stratification in density and velocity. Though, there is no estimate of these stratifications.

\noindent
Finally, the augmented model was introduced. The {\em conservativity} of this model avoid chossing a conservative path, introduced in \cite{dal1995definition}, to solve the numerical problem.

\section*{Acknowledgments}
The author warmly thanks P. Noble and J.P. Vila for their precious help.

\bibliographystyle{plain}

\bibliography{LWP_Two_Layer_RMonjarret}

\end{document}